\documentclass[11pt]{article}

\usepackage{amsthm}
\usepackage{amsmath}
\usepackage{amssymb}
\usepackage[margin=1in]{geometry}
\usepackage[backref=page]{hyperref}
\usepackage{graphicx}
\usepackage{booktabs}
\usepackage{enumitem}
\usepackage[outerbars,color]{changebar}
\usepackage{listings}

\hypersetup{
    colorlinks=true,     
    linkcolor=blue,      
    citecolor=blue,      
    filecolor=blue,      
    urlcolor=blue        
}

\newtheorem{proposition}{Proposition}
\newtheorem{theorem}{Theorem}
\newtheorem{lemma}{Lemma}

\newtheorem{corollary}{Corollary}

\newtheorem{conjecture}{Conjecture}

\newtheorem{question}{Question}

\newcommand{\cS}{\mathcal{S}}

\newcommand{\orthant}[2]{\RR_+^{#1,(#2)}}
\newcommand{\psdcone}[2]{\cS_+^{#1,(#2)}}

\newcommand{\ones}{1}

\renewcommand{\S}{\cS}

\newcommand{\psd}{\succeq}

\newcommand{\pd}{\succ}
\newcommand{\RR}{\mathbb{R}}

\newcommand{\tr}{\textup{tr}}

\DeclareMathOperator{\diag}{diag}

\title{A spectrahedral representation of the first derivative\\ relaxation of the positive semidefinite cone}

\author{James Saunderson\thanks{Department of Electrical and Computer Systems Engineering, Monash University, VIC 3800, Australia. Email: \texttt{james.saunderson{@}monash.edu}}}

\begin{document}
\maketitle

\begin{abstract}
	If $X$ is an $n\times n$ symmetric matrix, then the directional derivative of $X \mapsto \det(X)$ in the direction $I$ is the elementary 
symmetric polynomial of degree $n-1$ in the eigenvalues of $X$. This is a polynomial in the entries of $X$ with the property that it is hyperbolic
with respect to the direction $I$. The corresponding hyperbolicity cone is a relaxation of the positive semidefinite (PSD) cone
known as the first derivative relaxation (or Renegar derivative) of the PSD cone. A spectrahedal cone is a convex cone that has a representation as the intersection of a 
subspace with the cone of PSD matrices in some dimension. We show that the first derivative relaxation of the PSD cone is a spectrahedral cone, and give 
an explicit spectrahedral description of size $\binom{n+1}{2}-1$. The construction provides a new explicit example of a hyperbolicity cone that is 
also a spectrahedron. This is consistent with the generalized Lax conjecture, which conjectures that every hyperbolicity cone is a spectrahedron.

\end{abstract}

\section{Introduction}
\label{sec:intro}

\subsection{Preliminaries}

\paragraph{Hyperbolic polynomials, hyperbolicity cones, and spectrahedra}

A multivariate polynomial $p$, homogeneous of degree $d$ in $n$ variables, is
\emph{hyperbolic with respect to $e\in \RR^n$} if $p(e) \neq 0$ and for all
$x$, the univariate polynomial $t\mapsto p(x-te)$ has only real roots.
Associated with such a polynomial is a cone 
\[ \Lambda_{+}(p,e) = \{x\in \RR^n: \textup{all roots of $t\mapsto p(x-te)$ are non-negative}\}.\] 
A foundational result of G\r{a}rding~\cite{gaarding1959inequality} is that $\Lambda_{+}(p,e)$ is actually a
convex cone, called the \emph{closed hyperbolicity cone associated with $p$ and $e$}. 

For example $p(x) = \prod_{i=1}^{n}x_i$ is hyperbolic with respect to
$\ones_n$, the vector of all ones, and the corresponding closed hyperbolicity cone is the
non-negative orthant, $\RR_{+}^n$.  Similarly $p(X) = \det(X)$ (where $X$ is a
symmetric $n\times n$ matrix), is hyperbolic with respect to the identity
matrix $I$, and the corresponding closed hyperbolicity cone is the positive
semidefinite cone $\cS_{+}^n$.   

If a polynomial $p$ has a representation of the form
\begin{equation}
	\label{eq:def-det}
	 p(x) = \det\left(\textstyle{\sum_{i=1}^{n}}A_ix_i\right)
\end{equation}
for symmetric matrices $A_1,\ldots,A_n$, and there exists $e\in \RR^n$ such
that $\sum_{i=1}^{n}A_ie_i$ is positive definite, we say that $p$ has a
\emph{definite determinantal representation}.  In this case $p$ is hyperbolic
with respect to $e$. The associated closed hyperbolicity cone is
\begin{equation}	
\label{eq:spectrahedral} 
K = \bigg\{x\in \RR^n\;:\; \sum_{i=1}^{n}A_i x_i \psd 0\bigg\}
\end{equation}
where we write $X \psd 0$ to indicate that $X$ is positive semidefinite (and $X
\pd 0$ to indicate that $X$ is positive definite).  Such convex cones are
called \emph{spectrahedral cones}. If the matrices $A_1,A_2,\ldots,A_n$ are
$d\times d$ we call~\eqref{eq:spectrahedral} a \emph{spectrahedral
representation of size $d$}.

\paragraph{Derivative relaxations}
One way to produce new hyperbolic polynomials is to take directional
derivatives of hyperbolic polynomials in directions of
hyperbolicity~\cite[Section 3.10]{atiyah1970lacunas}, a  construction emphasized in the context of optimization by
Renegar~\cite{renegar2006hyperbolic}.  If $p$ has degree $d$ and is hyperbolic
with respect to $e$, then for $k=0,1,\ldots,d$, the $k$th directional
derivative in the direction $e$, i.e.,  
\[ D_e^{(k)}p(x) = \left.\frac{d^k}{dt^k}p(x+te)\right|_{t=0},\]
is also hyperbolic with respect to $e$. Moreover 
\[ \Lambda_{+}(D_e^{(k)}p,e) \supseteq \Lambda_{+}(D_e^{(k-1)}p,e) \supseteq \cdots \supseteq \Lambda_{+}(p,e)\] 
so the hyperbolicity cones of the directional derivatives form a sequence of
\emph{relaxations} of the original hyperbolcity cone. 

\begin{itemize}
	\item Suppose $p(x) = \prod_{i=1}^{n}x_i$ and $e=\ones_n$. Then, for $k=0,1,\ldots,n$, 
\[ D_{\ones_n}^{(k)}p(x) = {k!}e_{n-k}(x)\]
where $e_{n-k}$  is the elementary symmetric polynomial of degree $n-k$ in $n$ variables. 
We use the notation $\orthant{n}{k}$ for $\Lambda_+(e_{n-k},\ones_n)$, the closed hyperbolicity cone corresponding to $e_{n-k}$. 
	\item Suppose $p(X) = \det(X)$ is the determinant restricted to $n\times n$ symmetric matrices, 
	and $e=I_n$ is the $n\times n$ identity matrix. Then, for $k=0,1,\ldots,n$,  
\[ D_{I_n}^{(k)}p(X) = {k!}\,E_{n-k}(X) = {k!}\,e_{n-k}(\lambda(X))\]
where $E_{n-k}(X)$ is the elementary symmetric polynomial of degree $n-k$ in
the eigenvalues of $X$ or, equivalently, the coefficient of $t^k$ in
$\det(X+tI_n)$. We use the notation $\psdcone{n}{k}$ for $\Lambda_+(E_{n-k},I_n)$, the closed
hyperbolicity cone corresponding to $E_{n-k}$. We use the notation $\lambda(X)$ for the 
eigenvalues of a symmetric matrix $X$ ordered so that $|\lambda_1(X)| \geq |\lambda_2(X)| \geq \cdots \geq |\lambda_n(X)|$. 
We use this order so that $\lambda_i(X^2) = \lambda_i(X)^2$ for all $i$.
\end{itemize}


The focus of this paper is the cone $\psdcone{n}{1}$, the hyperbolicity cone
associated with $E_{n-1}$. In particular, we consider whether
$\psdcone{n}{1}$ can be expressed as a `slice' of some higher dimensional
positive semidefinite cone.  Such a description allows one
to reformulate hyperbolic programs with respect to $\psdcone{n}{1}$ (linear optimization over affine `slices' of $\psdcone{n}{1}$)
as
semidefinite programs.

\paragraph{Generalized Lax conjecture}

We have seen that every spectrahedral cone is a closed hyperbolicity cone. 
The \emph{generalized Lax conjecture} asks whether the converse holds, i.e.,
whether every closed hyperbolicity cone is also a spectrahedral cone.  The
original Lax conjecture, now a theorem due to Helton and
Vinnikov~\cite{helton2007linear} (see also~\cite{lewis2005lax}), states that if
$p$ is a trivariate polynomial, homogeneous of degree $d$, and  hyperbolic with
respect to $e\in \RR^3$, then $p$ has a definite determinantal representation.
While a direct generalization of this algebraic result does not hold in higher
dimensions~\cite{branden2011obstructions}, the following geometric conjecture
remains open.
\begin{conjecture}[Generalized Lax Conjecture (geometric version)]
\label{conj:geo}
Every closed hyperbolicity cone is spectrahedral.
\end{conjecture}
An equivalent algebraic formulation of this conjecture is as follows. 
\begin{conjecture}[Generalized Lax Conjecture (algebraic version)]
\label{conj:alg}
If $p$ is hyperbolic with respect to $e\in \RR^n$, then there exists a
polynomial $q$, hyperbolic with respect to $e\in \RR^n$, such that $qp$ has a
definite determinantal representation and $\Lambda_{+}(q,e)\supseteq \Lambda_+(p,e)$.
\end{conjecture}
The algebraic version of the conjecture implies the geometric version because
it implies the existence of a multiplier $q$ such
that the hyperbolicity cone associated with $qp$ is spectrahedral and 
$\Lambda_+(qp,e) = \Lambda_+(p,e)\cap \Lambda_+(q,e) = \Lambda_+(p,e)$. 
To see that the geometric version implies the algebraic version requires
more algebraic machinery, and is discussed, for instance, in~\cite[Section 2]{vinnikov2012lmi}.

\subsection{Main result: a spectrahedral representation of $\psdcone{n}{1}$}

In this paper, we show that $\psdcone{n}{1}$, the first derivative relaxation
of the positive semidefinite cone, is spectrahedral.  We give an explicit
spectrahedral representation of $\psdcone{n}{1}$ (see Theorem~\ref{thm:main} to follow).
Moreover, in Theorem~\ref{thm:main-alg} in Section~\ref{sec:pf} we find an explicit
hyperbolic polynomial $q$ such that $q(X)E_{n-1}(X)$ has a definite
determinantal representation and $\Lambda_{+}(q,I) \supseteq \psdcone{n}{1}$.  
\begin{theorem}
\label{thm:main}
	Let $d = \binom{n+1}{2}-1$ and let $B_1,\ldots,B_d$ be any basis for
the $d$-dimensional space of real symmetric $n\times n$ matrices with trace
zero.  If $\mathcal{B}(X)$ is the $d\times d$ symmetric matrix with $i,j$ entry
equal to $\tr(B_i X B_j)$ then 
	\begin{equation}	
	\label{eq:main-geo} \psdcone{n}{1} = \{X\in \S^n: \mathcal{B}(X) \psd 0\}.
	\end{equation}
\end{theorem}
Section~\ref{sec:pf} is devoted to the proof of this result. At this stage we
make a few remarks about the statement and some of its consequences.
\begin{itemize}
	\item The spectrahedral representation of $\psdcone{n}{1}$ in
Theorem~\ref{thm:main} has size $d = \binom{n+1}{2}-1 = \frac{1}{2}(n+2)(n-1)$.
This is about half the size of the smallest previously known \emph{projected}
spectrahedral representation of $\psdcone{n}{1}$, i.e., representation 
as the image of a spectrahedral cone under a linear map~\cite{saunderson2015polynomial}.
\item A straightforward extension of this result shows that if $p$ has a definite determinantal representation and $e$ is a direction of 
hyperbolicity for $p$, then the hyperbolicity cone associated with the directional derivative $D_ep$ is spectrahedral.
We discuss this in Section~\ref{sec:consequences}.
	\item It also follows from Theorem~\ref{thm:main} that $\orthant{n}{2}$, the second
derivative relaxation of the orthant in the direction $\ones_n$, has a spectrahedral
representation of size $\binom{n}{2}-1$. We discuss this in
Section~\ref{sec:consequences}.  This representation is significantly smaller
than the size $O(n^{n-3})$ representation constructed by
Br\"and\'en~\cite{branden2014hyperbolicity}, and about half the size of 
the smallest previously known projected spectrahedral
representation of $\orthant{n}{2}$~\cite{saunderson2015polynomial}.
\end{itemize}

\subsection{Related work}
\label{sec:related}
We briefly summarize related work on spectrahedral and projected spectrahedral
representations of the hyperbolicity cones $\orthant{n}{k}$ and
$\psdcone{n}{k}$.  Sanyal~\cite{sanyal2013derivative} showed that
$\orthant{n}{1}$ is spectrahedral by giving the following explicit definite
determinantal representation of $e_{n-1}(x)$, which we use repeatedly in the paper.
\begin{proposition}
	\label{prop:sanyal}
	If $\ones_n^\perp = \{x\in \RR^n\;:\; \ones_n^Tx = 0\}$, and $V_n$ is a
$n\times (n-1)$ matrix with columns spanning $\ones_n^\perp$, then there is a
positive constant $c$ such that 
	\[ c\,e_{n-1}(x) = \det(V_n^T\diag(x)V_n) \;\;\textup{and so}\;\;
\orthant{n}{1} = \{x\in \RR^n: V_n^T\diag(x)V_n \psd 0\}.\]
\end{proposition}
This representation is also implicit in the work of Choe, Oxley, Sokal, and
Wagner~\cite{choe2004homogeneous}.
Zinchenko~\cite{zinchenko2008hyperbolicity}, gave a projected spectrahedral
representation of $\orthant{n}{1}$.
Br\"and\'en~\cite{branden2014hyperbolicity}, established that each of the cones
$\orthant{n}{k}$ are spectrahedral by constructing graphs $G$ with edges
weighted by linear forms in $x$, such that the edge weighted Laplacian $L_G(x)$
is positive semidefinite if and only if $x\in \orthant{n}{k}$. Amini
showed that the hyperbolicity cones associated with
certain multivariate matching polynomials are
spectrahedral~\cite{amini2016spectrahedrality}, and used these to find
new spectrahedral representations of the cones $\orthant{n}{k}$ of size $\frac{(n-1){!}}{(k-1){!}}+1$.

Explicit projected spectrahedral representations of the cones $\psdcone{n}{k}$
of size $O(n^2\min\{k,n-k\})$ were given by Saunderson and
Parrilo~\cite{saunderson2015polynomial}, leaving open (except in the cases $k=n-2,n-1$) 
the question of whether these cones
are spectrahedra.  The main result of this paper is that $\psdcone{n}{1}$ is a
spectrahedron.

\section{Proof of Theorem~\ref{thm:main}}
\label{sec:pf}
 In this section we give two proofs of Theorem~\ref{thm:main}. The first proof is convex geometric in nature whereas the second is algebraic in nature.
 Both arguments are self-contained. We present the geometric argument first because it suggests the choice of multiplier $q$ for the algebraic argument.

 Both arguments take advantage of the fact that the cone $\psdcone{n}{1}$ satisfies
 $Q\psdcone{n}{1}Q^T = \psdcone{n}{1}$ for all $Q\in O(n)$. 
 One way to see this is to observe that the hyperbolic polynomial $E_{n-1}(X)$ that determines the cone satisfies 
 $E_{n-1}(QXQ^T) = E_{n-1}(X)$ for all $Q\in O(n)$ and the direction of hyperbolicity (the identity) is also invariant under this group action.

\subsection{Geometric argument}

  We begin by stating a slight reformulation of Sanyal's spectrahedral representation (Proposition~\ref{prop:sanyal}).
   \begin{proposition}
   \label{prop:sanyal-geo}
   Let $\ones_n^\perp = \{y\in \RR^n\;:\; \ones_n^Ty = 0\}$ be the subspace of $\RR^n$ orthogonal to $\ones_n$. Then
   \[ \orthant{n}{1} = \{x\in \RR^n\;:\; y^T\diag(x)y \geq 0\;\;\textup{for all $y\in \ones_n^\perp$}\}.\]
   \end{proposition} 
\begin{proof}  This follows from Proposition~\ref{prop:sanyal} since $V_n^T\diag(x)V_n \psd 0$ holds if and only if $u^TV_n^T\diag(x)V_nu \geq 0$ for all $u\in \RR^{n-1}$ which 
holds if and only if $y^T\diag(x)y \geq 0$ for all $y\in \ones_n^\perp$. 
\end{proof}

In this section we establish a `matrix' analogue of Proposition~\ref{prop:sanyal-geo}.
\begin{theorem}
	\label{thm:main-geo}
Let $I_n^\perp = \{Y\in \cS^n\;:\; \tr(Y) = 0\}$ be the subspace of $n\times n$ symmetric matrices with trace zero. Then
\begin{equation}
	\label{eq:main-geo} \psdcone{n}{1} = \{X\in \cS^n\;:\; \tr(YXY)\geq 0,\;\;\textup{for all $Y\in I_n^\perp$}\}.
\end{equation}
\end{theorem}
The concrete spectrahedral description given in Theorem~\ref{thm:main} follows immediately from Theorem~\ref{thm:main-geo}. 
Indeed if $B_1,B_2,\ldots,B_d$ are a basis for $I_n^\perp$ then 
an arbitrary $Y\in I_n^\perp$ can be written as $Y = \sum_{i=1}^{d} y_i B_i$. 
The condition $\tr(YXY) \geq 0$ for all $Y\in I_n^\perp$ is equivalent to 
\[ \sum_{i,j=1}^{d}y_iy_j\tr(B_iXB_j) \geq 0\;\;\textup{for all $y\in \RR^{d}$}\;\;\textup{which holds if and only if}\;\; \mathcal{B}(X) \psd 0.\]

\begin{proof}[{of Theorem~\ref{thm:main-geo}}]
The convex cone $\psdcone{n}{1}$ is invariant under the action
of the orthogonal group on $n\times n$ symmetric matrices by congruence transformations. 
Similarly, the convex cone
\[ \{X\in \cS^n\;:\; \tr(YXY) \geq 0 \quad\textup{for all $Y\in I_n^\perp$}\}\]
is invariant under the same action of the orthogonal group. This is because 
$X\in I_n^\perp$ if and only if $QXQ^T\in I^\perp$ for any orthogonal matrix $Q$. 

 Because of these invariance properties, the following (straightforward) result tells us 
 that we can establish Theorem~\ref{thm:main-geo} 
 by showing that the diagonal `slices' of these two convex cones agree.
 \begin{lemma}
 	\label{lem:orth-inv}
 Let $K_1,K_2\subset \cS^n$ be such that $QK_1Q^T = K_1$ for all $Q\in O(n)$
 and $QK_2Q^T = K_2$ for all $Q\in O(n)$. If 
 $\{x\in \RR^n\;:\; \diag(x)\in K_1\} = \{x\in \RR^n\;:\;\diag(x)\in K_2\}$
 then $K_1 = K_2$.
 \end{lemma} 
 \begin{proof}
 Assume that  
$X\in K_1$. Then there exists $Q$ such that 
 $QXQ^T = \diag(\lambda(X))$. Since $K_1$ is invariant under orthogonal congruence, $\diag(\lambda(X)) \in K_1$. 
 By assumption, it follows that $\diag(\lambda(X))\in K_2$. Since $K_2$ is invariant under orthogonal congruence,
 $X = Q^T\diag(\lambda(X))Q \in K_2$. This establishes that $K_1 \subseteq K_2$. Reversing the roles of $K_1$ and $K_2$ completes the argument. 

 \end{proof}

\paragraph{Relating the diagonal slices} To complete the proof of
Theorem~\ref{thm:main-geo}, it suffices (by Lemma~\ref{lem:orth-inv}) 
to show that the diagonal slices of the
left- and right-hand sides of~\eqref{eq:main-geo} are equal. Since the diagonal
slice of $\psdcone{n}{1}$ is $\orthant{n}{1}$,  it is enough (by Proposition~\ref{prop:sanyal-geo}) to establish the following
result.
\begin{lemma}
	\label{lem:orthantn1alt}
	\begin{multline*}
	 \{x\in \RR^{n}\;:\; \tr(Y\diag(x)Y) \geq 0 \;\;\textup{for all $Y\in I_n^\perp$}\} = \\
		\{x\in \RR^{n}\;:\; y^T\diag(x)y \geq 0 \;\;\textup{for all $y\in \ones_n^\perp$}\}.
	\end{multline*}
\end{lemma}
\begin{proof}
	Suppose that $\tr(Y\diag(x)Y) \geq 0$ for all $Y\in I_n^\perp$. Let
$y\in \ones_n^\perp$. Then $\diag(y)\in I_n^\perp$ and so it follows that
$\tr(\diag(y)\diag(x)\diag(y))  = y^T\diag(x)y \geq 0$. This shows that the
left hand side is a subset of the right hand side.

For the reverse inclusion suppose that $y^T\diag(x)y \geq 0$ for all $y\in
\ones_n^\perp$. Let $Y\in I_n^\perp$. Suppose the symmetric group on $n$ symbols, $S_n$,  acts on $\RR^n$ by permutations. 
Then for every $\sigma \in S_n$, we have that $\sigma \cdot \lambda(Y)\in \ones_n^\perp$ and thus  
\[ \tr(\diag(\sigma \cdot \lambda(Y^2))\diag(x)) = (\sigma \cdot\lambda(Y))^T\diag(x)(\sigma \cdot \lambda(Y))  \geq 0.\]
(Here we have used $\lambda_i(Y^2) = \lambda_i(Y)^2$, by our definition of $\lambda(\cdot)$.)

The diagonal of a symmetric matrix is a convex combination of permutations of
its eigenvalues, a result due to Schur~\cite{schur1923uber} (see also,
e.g.,~\cite{marshall1979inequalities}).  Hence $\diag(Y^2)$ is a convex
combination of permutations of $\lambda(Y^2)$, i.e.,
\[ \diag(Y^2) = \sum_{\sigma\in S_n} \eta_{\sigma}\, (\sigma\cdot \lambda(Y^2))\]
where the $\eta_\sigma$ satisfy
$\eta_{\sigma} \geq 0$ and $\sum_{\sigma \in S_n}\eta_\sigma =1$.  It then
follows that 
\[\tr(Y\diag(x)Y) =  \tr(\diag(Y^2)\diag(x)) = \sum_{\sigma\in S_n}\eta_{\sigma}\tr(\diag(\sigma\cdot\lambda(Y^2))\diag(x)) \geq 0.\]
This shows that the right hand side is a subset of the left hand side.
\end{proof}
This completes the proof of Theorem~\ref{thm:main-geo}.
\end{proof}

\subsection{Algebraic argument}

In this section, we establish the following algebraic version of Theorem~\ref{thm:main}.
\begin{theorem}
	\label{thm:main-alg}
	Let $n\geq 2$ and $B_1,\ldots,B_d$ be a basis for $I_n^\perp$, the subspace of $n\times n$ symmetric matrices with trace zero. Then there is a 
	positive constant $c$ (depending on the choice of basis) such that 
	\begin{enumerate}
		\item $q(X) = \prod_{1\leq i<j\leq n} (\lambda_i(X)+\lambda_j(X))$ is hyperbolic with respect to $I_n$;
		\item the hyperbolicity cone associated with $q$ satisfies 
		\[ \Lambda_{+}(q,I_n) = \{X\in \cS^n\;:\; \lambda_i(X) + \lambda_j(X) \geq 0\;\;\textup{for all $1\leq i<j\leq n$}\} \supseteq \psdcone{n}{1};\]
		\item $q(X)E_{n-1}(X)$ has a definite determinantal representation as
		\[ c\,q(X)E_{n-1}(X) = \det(\mathcal{B}(X)).\]
	\end{enumerate}
\end{theorem}
We remark that $q(X)$ is defined as a symmetric polynomial in the eigenvalues
of $X$, and so can be expressed as a polynomial in the entries of $X$.
Although our argument does not use this fact, it can be shown that $q(X) =
\det(\mathcal{L}_2(X))$ where $\mathcal{L}_2(X)$ is the \emph{second additive
compound matrix} of $X$~\cite{fiedler1974additive}.  This means that $q$ is not
only hyperbolic with respect to $I_n$, but also has a definite determinantal
representation.

\begin{proof}[{of Theorem~\ref{thm:main-alg}}]
{}
The three items in the statement of Theorem~\ref{thm:main-alg} are established
in the following three Lemmas (Lemmas~\ref{lem:alg1},~\ref{lem:alg2}, and~\ref{lem:alg3}).

\begin{lemma}
	\label{lem:alg1}
	If $q(X) = \prod_{1\leq i<j\leq n}(\lambda_i(X)+\lambda_j(X))$ then $q$ is hyperbolic with respect to $I_n$.
\end{lemma}
\begin{proof}
	First observe that $q(I_n) = 2^{\binom{n}{2}}\neq 0$. Moreover, for any real $t$, 
	\[ q(X-tI_n) = \prod_{1\leq i<j\leq n}(\lambda_i(X-tI_n) + \lambda_j(X-tI_n)) = \prod_{1\leq i<j\leq n}(\lambda_i(X) + \lambda_j(X) - 2t)\]
	which has $\binom{n}{2}$ real roots given by $\frac{1}{2}(\lambda_i(X)
+ \lambda_j(X))$ for $1\leq i<j\leq n$. Hence $q$ is hyperbolic with respect to
$I_n$. 
\end{proof}

\begin{lemma}
	\label{lem:alg2}
	If $n\geq 2$ then 
	\[ \Lambda_{+}(q,I_n) = \{X\in \cS^n\;:\; \lambda_i(X) + \lambda_j(X) \geq 0\;\;\textup{for all $1\leq i<j\leq n$}\}\supseteq \psdcone{n}{1}.\] 
\end{lemma}
 \begin{proof}
	Since the roots of $t\mapsto q(X-tI_n)$ are $\frac{1}{2}(\lambda_i(X) +
\lambda_j(X))$, the description of $\Lambda_+(q,I_n)$ is immediate. Both sides
of the inclusion are invariant under congruence by orthogonal matrices. By
Lemma~\ref{lem:orth-inv} it is enough to show that the inclusion holds
for the diagonal slices of both sides. Note that 
	\[ \{x\in \RR^n\;:\; \diag(x)\in \Lambda_{+}(q,I_n)\} = \{x\in \RR^{n}\;:\; x_i+x_j \geq 0\;\;\textup{for all $1\leq i<j\leq n$}\}.\]
	Hence it is enough to establish that
 	\begin{equation}
		\label{eq:incl}
		 \{x\in \RR^n\;:\; x_i+x_j \geq 0\;\;\textup{for all $1\leq i<j\leq n$}\} \supseteq \orthant{n}{1}.
	\end{equation}
	To do so, we use the characterization of $\orthant{n}{1}$ from Proposition~\ref{prop:sanyal-geo}. 
	This tells us that if $x\in \orthant{n}{1}$ then $v^T\diag(x)v = \sum_{\ell=1}^{n}x_\ell v_\ell^2 \geq 0$ 
	for all $v\in \ones_n^\perp$.
	In particular, let $v$ be the element of $\ones_n^\perp$ with $v_i=1$ and $v_j=-1$ and $v_k = 0$ for $k\notin\{i,j\}$. 
	Then, if $x\in \orthant{n}{1}$ it follows that 
	$\sum_{\ell=1}^{n}x_\ell v_\ell^2 = x_i + x_j \geq 0$. This completes the proof.
%
 \end{proof}

\begin{lemma}
	\label{lem:alg3}
If $B_1,\ldots,B_d$ is a basis for $I^\perp_n$, then there is a positive
constant $c$ (depending on the choice of basis) such that 
\[ c\,q(X)E_{n-1}(X) = \det(\mathcal{B}(X)).\]
\end{lemma}
\begin{proof}
Since both sides are invariant under orthogonal congruence, it is enough to
show that the identity holds for diagonal matrices. In other words, it is
enough to show that 
\[ c\prod_{1\leq i<j\leq n}(x_i+x_j) e_{n-1}(x) = \det(\mathcal{B}(\diag(x))).\] 
Since a change of basis for the subspace of symmetric matrices with trace zero
only changes $\det(\mathcal{B}(X))$ by a positive constant (which is one if the
change of basis is orthogonal with respect to the trace inner product), it is
enough to choose a particular basis for the subspace of symmetric matrices with
trace zero, and show that the identity holds for a particular constant.  

Let $v_1,v_2,\ldots,v_{n-1}$ be a basis for $\ones_n^\perp = \{x\in \RR^{n}\;:\;
\sum_{i=1}^{n}x_i = 0\}$.  Let $M_{ij}$ be the $n\times n$ matrix with a one in the $(i,j)$ and the $(j,i)$ entry, and zeros elsewhere.
Clearly the $M_{ij}$ for $1\leq i<j\leq n$ form a basis for the subspace of symmetric matrices with zero
diagonal. Together $\diag(v_1),\diag(v_2),\ldots,\diag(v_{n-1})$ and $M_{ij}$
for $1\leq i<j\leq n$ form a basis for the subspace of symmetric matrices with
trace zero.

Using this basis we evaluate the matrix $\mathcal{B}(\diag(x))$. We note that
\begin{align*}
	 \tr(\diag(v_i)\diag(x)\diag(v_j)) & = v_i^T\diag(x)v_j\quad\textup{for $1\leq i,j\leq n$}\\
	\tr(\diag(v_i)\diag(x)M_{jk}) & = 0\quad\textup{for all $1\leq i\leq n$ and $1\leq j<k\leq n$}
\end{align*}
since $M_{jk}$ has zero diagonal, and that
\[ 
	\tr(M_{ij}\diag(x)M_{k\ell})  = \begin{cases} x_i+x_j & \textup{if $i=k$ and $j=\ell$}\\ 0 & \textup{otherwise}\end{cases}\]
for all $1\leq i<j\leq n$ and $1\leq k<\ell\leq n$.
This means that $\mathcal{B}(\diag(x))$ is block diagonal, and so
\begin{equation}	
\label{eq:block-det} 
\det(\mathcal{B}(\diag(x))) = \prod_{1\leq i<j\leq n}(x_i+x_j) \det(V_n^T\diag(x)V_n)
\end{equation}
where $V_n$ is the $n\times (n-1)$ matrix with columns $v_1,v_2,\ldots,v_n$. By
Proposition~\ref{prop:sanyal}, there is a positive constant $c$ such that 
\begin{equation}	
\label{eq:sanyal} \det(V_n^T\diag(x)V_n) = c\,e_{n-1}(x),
\end{equation}
Combining~\eqref{eq:block-det} and~\eqref{eq:sanyal} gives the stated result. 
\end{proof}
This completes the proof of Theorem~\ref{thm:main-alg}. 
\end{proof}


\section{Discussion}
\label{sec:discussion}
\subsection{Consequences of Theorem~\ref{thm:main}}
\label{sec:consequences}
A straightforward consequence of Theorem~\ref{thm:main} is that if $p$ has a
definite determinantal representation, and $e$ is a direction of hyperbolicity
for $p$, then the hyperbolicity cone associated with the directional derivative
$D_ep$ is spectrahedral.
\begin{corollary}
	\label{cor:spec}
	If $p(x) = \det(\sum_{i=1}^{n}A_ix_i)$ for symmetric $\ell\times \ell$
matrices $A_1,\ldots,A_n$, and $A_0 = \sum_{i=1}^{m}A_i e_i$ is positive
definite, then $\Lambda_+(D_ep,e)$ has a spectrahedral representation of size
$\binom{\ell+1}{2}-1$. 
\end{corollary}
\begin{proof}
The hyperbolicity cone $\Lambda_+(D_ep,e)$ can be expressed as 
\[ \Lambda_{+}(D_ep,e)= \bigg\{x\in \RR^n\;:\;\sum_{i=1}^{n}A_0^{-1/2}A_iA_0^{-1/2}x_i \in \psdcone{n}{1}\bigg\}.\]
(see, e.g.,~\cite[Proposition 4]{saunderson2015polynomial}). Applying Theorem~\ref{thm:main} then gives 
\[ \Lambda_{+}(D_ep,e)	= \bigg\{x\in \RR^n\;:\;\mathcal{B}\left(\sum_{i=1}^{n}A_0^{-1/2}A_iA_0^{-1/2}x_i\right)\psd 0\bigg\}.\] 
\end{proof}
Our main result also yields a spectrahedral representation of $\orthant{n}{2}$,
the second  derivative relaxation of the non-negative orthant, of size
$\binom{n}{2}-1$. This is, in fact, a special case of Corollary~\ref{cor:spec}.
In the statement below, $V_n$ is any $n\times (n-1)$ matrix with columns that
span $\ones_n^\perp$.
 \begin{corollary}
 	\label{cor:e2}
	The hyperbolicity cone $\orthant{n}{2}$ has a spectrahedral
representation of size $\binom{n}{2}-1$ given by
 	\[ \orthant{n}{2} = \{x\in \RR^{n}\;:\; \mathcal{B}(V_n^T\diag(x)V_n) \psd 0\}.\]
 \end{corollary}
\begin{proof}
	First, we use the fact that  $\orthant{n}{2} = \Lambda_{+}(D_{\ones_n}
e_{n-1},\ones_n)$.  Then, by Sanyal's result (Proposition~\ref{prop:sanyal}),
we know that $e_{n-1}(x)$ has a definite determinantal representation.  The
stated result then follows directly from Corollary~\ref{cor:spec} with
polynomial $p = e_{n-1}$ and direction $e = \ones_n$. 
\end{proof}

\subsection{Questions}

\paragraph{Constructing spectrahedral representations}
It is natural to ask for which values of $k$ the  cones $\psdcone{n}{k}$ are spectrahedral. 
Our main result shows that $\psdcone{n}{1}$ has a spectrahedral representation of size $d = \binom{n+1}{2}-1$. 
The only other cases for which spectrahedral representations are known are the straightforward cases $k=n-1$ and $k=n-2$. 
If $k=n-1$ then 
\[ \psdcone{n}{n-1} = \{X\in \cS^n\;:\;\tr(X) \geq  0\}\]
is a spectrahedron (with a representation of size $1$). 
Since $\psdcone{n}{n-2}$ is a quadratic cone, it is a spectrahedron. To give an explicit representation, 
let $d = \binom{n+1}{2}-1$ and $B_1,B_2,\ldots,B_{d}$ be an \emph{orthonormal} basis
(with respect to the trace inner product) for the subspace $I_n^\perp$. Now 
$X\in \psdcone{n}{n-2}$ if and only if (see, e.g.,~\cite[Section 5.1]{saunderson2015polynomial})
\begin{equation}	
\label{eq:Snn2}
 \tr(X) \geq 0\;\;\textup{and}\;\;
\tr(X)^2 - \tr(X^2) = \left[\sqrt{\frac{n-1}{n}}\tr(X)\right]^2 - \sum_{i=1}^{d}\tr(B_iX)^2 \geq 0.
\end{equation}
By a well-known spectrahedral representation of the second-order cone, \eqref{eq:Snn2} holds if and only if 
\begin{equation}	
\label{eq:quad} \sqrt{\frac{n-1}{n}}\tr(X)I_{d} +  \begin{bmatrix} \tr(B_1X) & \tr(B_2X) & \tr(B_3X)&  \cdots & \tr(B_dX)\\
	\tr(B_2X) &  - \tr(B_1X)& 0  &\cdots &0\\
	\tr(B_3X)& 0 &  - \tr(B_1X) & \cdots & 0\\
	\vdots & \vdots & \vdots & \ddots & \vdots\\
	\tr(B_d X) & 0 &0 &\cdots &  - \tr(B_1X)\end{bmatrix} \psd 0.
\end{equation}
So we see that $\psdcone{n}{n-2}$ has a spectrahedral representation of size $d= \binom{n+1}{2}-1$. 
At this stage, it is unclear how to extend the approach in this paper to the remaining cases.
\begin{question}
	Are the cones $\psdcone{n}{k}$ spectrahedral for $k=2,3,\ldots,n-3$?
\end{question}	
At first glance, it may seem that Corollary~\ref{cor:spec} allows us to
construct a spectrahedral representation for $\psdcone{n}{2}$ from a
spectrahedral representation for $\psdcone{n}{1}$. However, this is not the
case.  To apply Corollary~\ref{cor:spec} to this situation, we would need a
definite determinantal representation of $E_{n-1}(X)$, which our main result
(Theorem~\ref{thm:main}) does not provide.

\paragraph{Lower bounds on size}

Another natural question concerns the size of spectrahedral representations of
hyperbolicity cones. Given a hyperbolicity cone $K$, there is a unique (up to
scaling) hyperbolic polynomial $p$ of smallest degree $d$ that vanishes on the
boundary of $K$~(see, e.g.,~\cite{kummer2016two}). Clearly any spectrahedral
representation must have size at least $d$, but it seems that in some cases the
smallest spectrahedral representation (if it exists at all) must have larger
size. 
\begin{question}
	Is there a spectrahedral representation of $\psdcone{n}{1}$ with size smaller than $\binom{n+1}{2}-1$?
\end{question}
Recently, there has been considerable interest in developing methods for
producing lower bounds on the size of projected spectrahedral descriptions of
convex sets (see, e.g.,~\cite{fawzi2015positive})	. There has been much
less development in the case of lower bounds on the size of spectrahedral
descriptions. The main work in this direction is due to
Kummer~\cite{kummer2016two}.  For instance it follows from~\cite[Theorem
1]{kummer2016two} that any spectrahedral representation of the quadratic cone
$\psdcone{n}{n-2}$ must have size at least
$\frac{1}{2}\left[\binom{n+1}{2}-1\right]$.  Furthermore, in the special case
that $\binom{n+1}{2} - 1 = 2^{k}+1$ for some $k$ (which occurs if $n=3$ and
$k=2$ or $n=4$ and $k=3$) then Kummer's work shows that any spectrahedral
representation of $\psdcone{n}{n-2}$ must have size at least
$\binom{n+1}{2}-1$. This establishes that the construction in~\eqref{eq:quad}
is optimal when $n=3$ and $n=4$. Furthermore, in the case $n=3$ we have that
$\psdcone{n}{1} = \psdcone{n}{n-2}$. Hence  our spectrahedral representation
for $\psdcone{n}{1}$ is also optimal if $n=3$.

\section*{Acknowledgments}
I would like to thank Hamza Fawzi for providing very helpful feedback on a draft of this paper.

\bibliographystyle{alpha}
\bibliography{drpsd}

\end{document}